\documentclass[11pt, english]{amsart}
\usepackage{amssymb,amsthm,amsmath,amstext,graphicx}
\usepackage{mathrsfs}
\usepackage{bm}
\usepackage[utf8]{inputenc}
\usepackage{amscd}
\usepackage{youngtab}
\usepackage[all]{xy}
\usepackage{comment}

\makeatletter
\ifnum\@ptsize=0  \fi \ifnum\@ptsize=2  \fi 

\title{Equivariant spaces \\ of matrices of constant rank}
\author{J.M. Landsberg, L. Manivel}
\date{\today}

\address{Department of Mathematics, Texas A\&M University, College Station, TX 77843-3368, USA}
\email{jml@math.tamu.edu}
\address{Institut de Math\'ematiques de Toulouse ; UMR 5219, Universit\'e de Toulouse \& CNRS, F-31062 Toulouse Cedex 9, France}
\email{manivel@math.cnrs.fr}

\theoremstyle{plain}
\newtheorem{theorem}{Theorem}
\newtheorem{conjecture}{Conjecture}
\newtheorem{prop}[theorem]{Proposition}
\newtheorem{lemma}[theorem]{Lemma}
\newtheorem{coro}[theorem]{Corollary}
\newtheorem{definition}[theorem]{Definition}
\newtheorem{question}[theorem]{Question}

\def\AA{{\mathbb{A}}}
\def\CC{{\mathbb{C}}}
\def\RR{{\mathbb{R}}}
\def\OO{{\mathbb{O}}}
\def\HH{{\mathbb{H}}}
\def\PP{{\mathbb{P}}}
\def\QQ{{\mathbb{Q}}}\def\ZZ{{\mathbb{Z}}}
\def\SS{{\mathbb{S}}}
\def\cO{{\mathcal{O}}}
\def\cE{{\mathcal{E}}}\def\Ext{{\mathcal Ext}}
\def\cF{{\mathcal{F}}}\def\cX{{\mathcal{X}}}
\def\cG{{\mathcal{G}}}\def\cK{{\mathcal{K}}}
\def\cL{{\mathcal{L}}}\def\cI{{\mathcal{I}}}
\def\cN{{\mathcal{N}}}\def\cR{{\mathcal{R}}}
\def\cS{{\mathcal{S}}}\def\cU{{\mathcal{U}}}
\def\cC{{\mathcal{C}}}\def\cQ{{\mathcal{Q}}}
\def\cV{{\mathcal{V}}}\def\cM{{\mathcal{M}}}
\def\ra{{\rightarrow}}
\def\lra{{\longrightarrow}}
\def\ft{{\mathfrak t}}\def\ff{{\mathfrak f}}\def\fs{{\mathfrak s}}
\def\fc{{\mathfrak c}}\def\fe{{\mathfrak e}}
\def\fg{{\mathfrak g}}\def\fp{\mathfrak{p}}\def\fk{\mathfrak{k}}
\def\fso{\mathfrak{so}}\def\fz{\mathfrak{z}}\def\fl{\mathfrak{l}}
\def\fsp{\mathfrak{sp}}
\def\fsl{\mathfrak{sl}}\def\fgl{\mathfrak{gl}}

\let\Iff\iff
\def\iff{\Leftrightarrow}

\def\om{\omega}

\newcommand\scalemath[2]{\scalebox{#1}{\mbox{\ensuremath{\displaystyle #2}}}}

\usepackage{xcolor}
\def\lau#1{\textcolor{green}{\;{\bf Laurent:} #1 {\bf }}}

\def\intprod{\negthinspace
\mathbin{\raisebox{.4ex}{\hbox{\vrule height .5pt width 4pt depth 0pt %
          \vrule height 4pt width .5pt depth 0pt}}}}
 
\def\trank{\text{rank}}

\def\bv{\mathbf{v}}

\def\BC{\mathbb C}\def\BS{\mathbb S}

\def\BP{\mathbb P}
\def\pp#1{\mathbb P^{#1}}
\def\fz{\mathfrak z}
\def\fgl{\mathfrak g\mathfrak l}

\def\pp#1{{\mathbb P}^{#1}}
\def\tdim{{\rm dim}}

\def\ww{\wedge}

\def\cI{{\mathcal I}}

\def\cE{{\mathcal E}}
\def\cF{{\mathcal F}}
\def\cG{{\mathcal G}}
\def\cR{{\mathcal R}}
\def\cS{{\mathcal S}}\def\cN{{\mathcal N}}
\def\cL{{\mathcal L}}

\def\cO{{\mathcal O}}
\def\CC{\mathbb C}
\def\RR{\mathbb R}
\def\HH{\mathbb H}
\def\AA{{\mathbb A}}

\def\OO{\mathbb O}

\def\ZZ{\mathbb Z}
\def\SS{\mathbb S}

\def\11{\mathbf 1}
\def\PP{\mathbb P}
\def\QQ{\mathbb Q}

\def\fs{{\mathfrak s}}
\def\fsl{{\mathfrak {sl}}}
\def\fsp{{\mathfrak {sp}}}

\def\fso{{\mathfrak {so}}}
\def\fe{{\mathfrak e}}

\def\ff{{\mathfrak f}}

\def\fz{{\mathfrak z}}

\def\fg{{\mathfrak g}}

\def\fp{{\mathfrak p}}
\def\fk{{\mathfrak k}}
\def\ft{{\mathfrak t}}
\def\fl{{\mathfrak l}}
 
\def\a{\alpha}

\def\s{\sigma}
\def\kk{\kappa}

\def\d{\delta}

\def\ot{{\mathord{ \otimes } }}
\def\op{{\mathord{\,\oplus }\,}}

\def\lra{{\mathord{\;\longrightarrow\;}}}
\def\ra{{\mathord{\;\rightarrow\;}}}

\def\dim{{\rm dim}\;}
\def\La#1{\Lambda^{#1}}

\def\cV{{\underline{V}}}

\def\fgl{\frak g\frak l}\def\fsl{\frak s\frak l}

\def\op{\oplus}

\def\ff#1{\Bbb F\Bbb F^{#1}}

\def\op{\oplus}

\def\ff{\mathfrak f}

\def\ul{\underline}

\def\s{\sigma}
\def\t{\tau}

\def\a{\alpha}

\def\fs{\mathfrak  s}

\def\fl{\mathfrak  l}

\def\fso{\frak{so}}

\def\BP{\mathbb  P}
\def\BC{\mathbb  C}

\def\pp#1{\mathbb  P^{#1}}
\def\cC{\mathcal  C}

\def\BS{\mathbb  S}

\def\fp{\mathfrak  p}

\def\ci{\mathcal  I}

\def\cQ{\mathcal  Q}

\def\fg{\mathfrak  g}

\def\hd{, \dotsc ,}

\def\La#1{\Lambda^{#1}}

\def\pp#1{\mathbb  P^{#1}}

\def\ur{\underline{\mathbf{R}}}

\def\ra{\rightarrow}

\def\tend{\operatorname{End}}
\def\tim{\operatorname{Im}}
\def\tdim{\operatorname{dim}}
\def\tker{\operatorname{ker}}

\def\tmod{\operatorname{mod}}

\def\thom{\operatorname{Hom}}
\def\trank{\operatorname{rank}}

\def\ww{\wedge}

\def\bbb{{\mathbf{b}}}

\def\be{\begin{equation}}
\def\ene{\end{equation}}
\def\aaa{{\mathbf{a}}}
\def\bbb{{\mathbf{b}}}
\def\ccc{{\mathbf{c}}}

\def\trank{\mathbf{R}}

\newcommand{\isom}{\cong}

\def\rank{\operatorname{rank}}

\def\cK{{\mathcal K}}

\def\trank{{\mathrm {rank}}}

\def\aaa{\mathbf{a}}
\def\bbb{\mathbf{b}}
\def\ccc{\mathbf{c}}

\def\bv{\mathbf{v}}

\theoremstyle{remark}
\newtheorem{remark}[theorem]{Remark}
\newtheorem{example}[theorem]{Example}

\begin{document}

\maketitle

\begin{abstract}
We  {use} representation theory  { to construct } spaces of matrices of constant
rank {.} These spaces are parametrized by the natural representation of the general 
linear group or the symplectic group. We  {present variants}   of this idea,  {with} more complicated 
representations,  {and others with} the orthogonal group. Our spaces of matrices correspond to vector bundles which are 
homogeneous  {but} sometimes   admit deformations to  {non-homogeneous} vector bundles, showing that 
 {these spaces of} matrices  {sometimes}  admit  {large} families of deformations.
\end{abstract}

\section{Introduction} 

It is a classical problem in both algebraic geometry and linear algebra  to construct linear spaces of matrices of 
constant rank (constant outside the origin).  
 {This problem is presented in the language
of algebraic geometry in} \cite{MR954659} and  {most
work on the problem since then has  used this language
and the tools it brings with it, including this paper}. 

There is a strong 
relationship with the study of vector bundles on projective space, another classical topic 
that attracted considerable attention. Indeed, a vector space of dimension $n+1$ of matrices
of size $a\times b$ can be seen as a matrix with linear entries, or equivalently as a 
morphism of sheaves $\psi:  {\cO^{\op a}_{\pp n}\lra \cO(1)^{\op b}_{\pp n}}$. If the rank is constant, equal to $r$, the image of this 
morphism is a vector bundle $\cE$ of rank $r$. Letting $\cK$ and $\cC$ denote the kernel and cokernel bundles,
we get the diagram below, 
where diagonals are short exact sequences. The vector bundle $\cE$ has very special properties, in particular:
\begin{itemize}
    \item $\cE$ and $\cE^\vee(1)$ are generated by global sections;
    \item as a consequence, $\cE$ is uniform, in the sense that its restriction to every line $L\subset\PP^n$ 
    splits in the same way:
    $$\cE_{|L}\simeq \cO_L(1)^{\oplus c_1(E)}\oplus \cO_L^{\oplus (r-c_1(E))}.$$
\end{itemize}

\smallskip
Uniform vector bundles have been classified up to rank $r\le n+1$: in this range, for $n\ge 3$ they are sums of 
line bundles and the tautological quotient bundle $\cQ$ or its dual (see \cite{MR3536970} and references therein; in that paper the same result is conjectured to hold for $n\ge 5$ and $r<2n$).   The general philosophy is that there should exist very few 
uniform vector bundles on $\PP^n$ of small rank, but they are easier to construct when the 
rank is large. Conversely, if $\cE$ is a rank $r$ vector bundle on $\PP^n$, such that 
$\cE$ and $ \cE^\vee(1)$ are generated by global sections (so that in particular $\cE$ is uniform), 
the natural morphism
$$\psi_\cE: H^0(\PP^n,\cE)\otimes\cO_{\PP^n}\longrightarrow H^0(\PP^n,\cE^\vee(1))^\vee\otimes\cO_{\PP^n}(1)$$
has constant rank $r$. 

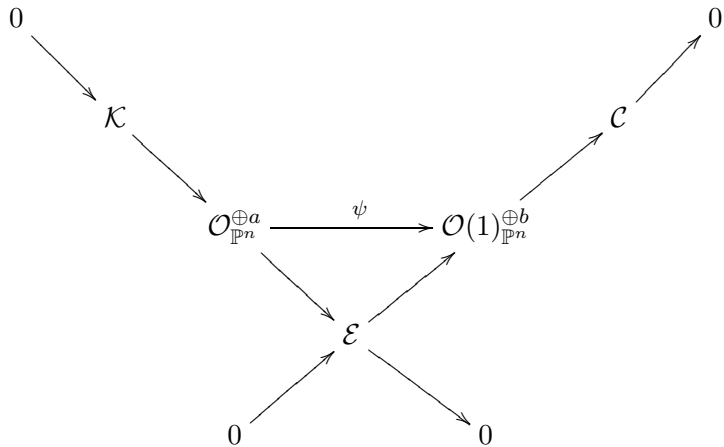
$$\xymatrix{0\ar[rd] & & & & & & 0\\
& \cK \ar[rd] &&& & \cC\ar[ru] & \\
& &  {\cO^{\op a}_{\PP^n}}\ar[rr]^\psi\ar[rd] & & {\cO(1)^{\op b}_{\PP^n}}\ar[ru] & &\\
& & & \cE\ar[ru]\ar[rd]& && \\
 & &0\ar[ru] & & 0 & &}$$
 
\medskip\noindent {\it Examples}.
\begin{enumerate}
\item  If $\cE$ is a sum of line bundles, one gets what is called in \cite{MR587090}
a compression space. 

\item If $\cE=\cQ$, we have $H^0(\PP^n,\cE)=V$ and $H^0(\PP^n,\cE^\vee(1))^\vee=
\wedge^2V$. In this case the twist $$\psi(-1) : 
 {V \ot \cO_{\pp n}(-1)\lra \wedge^2V
\ot \cO_{\pp n}}$$ is the obvious vector bundle map which at 
$[v]\in\PP^n$ sends $w\otimes v$ to $w\wedge v$. More generally,   let $\cE=\wedge^p\cQ$ for some $p>0$, 
and we get the similar morphisms $\psi(-1) : 
 {\wedge^pV\ot\cO_{\pp n}(-1)\lra \wedge^{p+1}V
\ot \cO_{\pp n}}$ from which the Koszul complex 
is constructed. 

\end{enumerate}

 \medskip This last example is classical and dates at least back to Westwick \cite{MR878293}.
 The main goal of this paper is to  {present} a wide generalization by considering maps $\psi :  {U\ot \cO_{\BP V}(-1)\lra W\ot \cO_{\BP V}}$,
 where $U, W$ are $G$-modules (that we will suppose irreducible, for simplicity) of a classical group  {$G\subseteq GL(V)$}, 
 with a nonzero equivariant morphism $\Psi : U\otimes V\lra W$ and $V$ denotes the natural representation. 
 
 \medskip\noindent {\bf Theorem}. {\it If $G$ is a general linear group, or a symplectic group, 
 then $\psi$  defines a linear
 space of matrices of constant rank. 
 
 The associated bundle $\cE$ is homogeneous and  {admits an
 explicit description via   a module for the  
 parabolic subgroup preserving a highest weight line in $V$.} }

 \smallskip
 The bundle $\cE$ may arise from an arbitrarily long sequence of nontrivial extensions of explicit 
 completely reducible homogeneous bundles, so its structure can be quite complicated. 
 Being homogeneous, we expected it to be rigid, as is typically the case of $\cE=\wedge^p\cQ$
 and of the first cases we checked.  {However},  very quickly one  {obtains non-rigid} bundles, 
 whose deformation spaces can have  {large} dimension.  {For} the case we discuss in detail, we get vector 
 bundles on $\PP^n$ with $O(n^4)$  {dimensions
 worth of} moduli, see Proposition \ref{moduli}. As a consequence,  $\psi$ may 
 be deformed into non-equivariant  linear spaces of matrices of constant rank. 
 
  {\medskip
   We make the following   simple observation. 
 
\begin{prop}\label{constant}
If $G=GL(V)$ or $G=Sp(V)$,   and $M,N$ are irreducible $G$-modules with $V\subset M^\vee\ot N$, then the 
image of the morphism $\psi: V\lra Hom(M,N)$ is a linear space 
of   constant rank.
\end{prop}

\proof This immediately follows from the equivariance of the morphism, and   that $G$ acts 
transitively on $\PP(V)$. \qed 
 

}
 
 \medskip
 \noindent {\it Overview:} After a section of preliminaries, in section 3 
 we revisit Westwick's examples and 
  {give a representation-theoretic discussion of the problem. }
 In section 4, focusing on the general linear group  we discuss the structure of the 
 associated bundle $\cE$; we explain why we get a rigid bundle in a simple case, and  a bundle with a 
 very  {large} deformation space in a slightly more complicated situation. We also discuss how to extend  
  {Proposition \ref{constant}} beyond the natural representation. In section 5 we observe that it can be extended to a simple 
 statement that allows  {one}  to construct many spaces of matrices of constant rank starting from a given one. 
 Then we consider the case of the symplectic group, and    {explicitly construct} a six dimensional space of 
 $14\times 14$ matrices of rank nine. Finally we discuss what can be obtained for the orthogonal groups, 
 and show that  {one} can at least construct  { spaces of} bounded rank; either from the natural 
 representation in section 7, or the spin representations in section 8. 

\medskip 
The reader may have observed that our constructions are closely connected to the general theory of 
Steiner vector bundles on projective spaces \cite{MR1240599}. In
 {a subsequent paper we present   constructions of 
 Steiner   bundles giving rise to new classes
 of spaces of constant rank.}

\medskip\noindent {\it Acknowledgements}. Most of the results of this paper were obtained during  
 {Landsberg's} stay in Toulouse, funded by the LabEx CIMI. L. Manivel is also supported by the ANR project
FanoHK, grant ANR-20-CE40-0023.  {
Landsberg supported by NSF grant AF-2203618.} We warmly thank Rosa Miro Roig for useful comments. 


\section{Preliminaries}

\subsection{Notation}

We work exclusively over the complex numbers.

$V$ is a complex vector space of dimension $\bv=n+1$.
We give $V$ basis $v_0\hd v_{n}$ and dual basis $\a^0\hd \a^n$.

For $\pi=(p_1\hd p_{n+1})$ a non-increasing sequence of integers, 
$S_\pi V$ denotes the corresponding irreducible $GL(V)$-module.
If $\pi$ contains repeated entries, we use exponents to denote them, e.g.
for $(1,1,1,0,0,-1)$  we write $(1^3,0^2,-1)$.

For a vector space $W$, we let $\ul{W}:=W\ot \cO_{\BP V}$ denote the corresponding trivial vector bundle
on $\BP V$  when it is clear from the context which projective space we are taking
the  bundle over. 

If $\cE,\cF\subseteq \ul W$, then $\cE\cF\subset \ul{S^2W}$ denotes the image of the multiplication
map.

\subsection{Basic spaces}\label{basicsect}
There are several ways to avoid redundancies in the classification of spaces of bounded and
constant rank.

A classical and essentially understood class of spaces of bounded rank are the {\it compression spaces}, those
spaces of $\bbb\times \ccc$ matrices 
that in some choice of bases have a zero in the lower $\bbb-r_1\times \ccc-r_2$ block. Such spaces
are of bounded rank $r_1+r_2$. 
Compression spaces  can have constant rank, e.g.,
$$
\begin{pmatrix} 0& x_1&\cdots & x_k\\ x_1& & &\\ \vdots & & & \\ x_k & & &\end{pmatrix}
$$
but note this is a subspace of the direct sum of two spaces of bounded rank one. To avoid
this, call a space of the form 
$$
\begin{pmatrix} M& 0\\ 0 & M'\end{pmatrix}
$$
{\it split} and following \cite{MR954659}, say a space is {\it strongly indecomposable} if it is not the projection
$B'\ot C'\ra B\ot C$
of a split space of the same rank. A space that is strongly indecomposable 
has associated  $\cE$   indecomposable  \cite{MR954659}.


\subsection{Rank criticality}

Maximal spaces of matrices of constant, or more generally bounded  {rank}, are particularly interesting. 

\begin{definition}\cite{MR2268360}\label{rcdef} 
A space of bounded rank $V\subset \thom(B,C)$  is {\it rank-critical} if 
   any subspace of $\thom(B,C)$ that strictly contains $V$  contains morphisms of larger rank. 
\end{definition}

 Eisenbud and Harris \cite{MR954659} call rank critical spaces {\it unliftable}. 
Draisma gives an easily checked sufficient condition for rank criticality:
 for  $L\subset Hom(U,W)$  a linear space  of morphisms
of generic rank $r$, define  the space of {\it rank neutral directions} 
\begin{align*}RND(L):&=\{ B\in Hom(U,W), \; B(Ker(A))\subset Im(A) \; \forall A\in L, rank(A)=r\}\\
&=\bigcap_{A\in L, rank(A)=r} \hat T_A\sigma_r(Seg(\BP U\times\BP W)).
\end{align*}
Applying \cite[Prop. 8]{MR2268360}
to a slightly more general situation than \cite[Prop. 3]{MR2268360}, $RND(L)$ always contains $L$ and in case
of equality, $L$ is  rank critical.   If there is a group $G$ acting on the set up  and preserving $L$, then 
it must also preserve $RND(L)$.

\section{Equivariant morphisms of constant or bounded rank}

\subsection{A classical example revisited}
The prototypical equivariant morphisms of constant rank appear in the Koszul complex, as the 
morphisms  
$$\psi : V\lra Hom(\wedge^kV,\wedge^{k+1}V).
$$ 
  For any nonzero vector $v$, the kernel (resp.   image) of $\psi_v$ 
is $\wedge^{k-1}V\wedge v$ (resp. $\wedge^{k}V\wedge v$). In other words,  {let   $\cQ$ be the tautological quotient bundle
on $\BP V$,}  we have exact sequences 
$$\xymatrix{ \wedge^{k-1}\cQ(-1) \ar@{->}[rd] & & & & \wedge^{k+1}\cQ(1)\\
 &  \ul{\wedge^{k}V}  \ar@{->}[rd]\ar@{->}[rr]^{\psi} & & \ul{\wedge^{k+1}V}(1) \ar@{->}[ru] & \\
  & & \wedge^{k}\cQ  {.} \ar@{->}[ru] & & 
}$$
So $\psi$ has constant rank and the vector bundle $\cE=\wedge^k\cQ$ is uniform. 
Here we slightly abuse notation, using $\psi$ to denote both the map between vector bundles and the
inclusion of $V$ into a space of homomorphisms.

All this is well-known, but we   add the following observation:

\begin{prop}\label{rank-critical}
$\psi(V)$ is rank-critical.
\end{prop}
 
\proof We apply the results  of \cite{MR2268360} discussed in \S\ref{basicsect}.
Here $U=\wedge^kV$,  $W=\wedge^{k+1}V$ and $L=\psi(V)$, and we assume $k\leq \frac n2$
to avoid redundancy. 
Then $Hom(U,W)=U^\vee\ot W$ has the following $GL(V)$-module decomposition:
$$ Hom(U,W)=\bigoplus_{k\geq a\ge 0}S_{1^{a+1 },0^{n-2a},-1^{a}}V.
$$
   Since there are no multiplicities, we are reduced to proving that $S_{1^{a },0^{n-2a},-1^{a+1}}V$ cannot be contained in $RND(L)$ when 
$a>0$. Equivalently, we need to check that a highest weight vector in $S_{1^{a+1}0^{n-2a},-1^{a}}V$ cannot be contained in $RND(L)$. 
A highest weight vector is   given by
$$
\sum_{|I|=k-a, a< i_1< \cdots < i_{a-k}<n+1-a}
v_0\wedge\cdots \wedge v_a\ww v_I\otimes \a_I\ww \a_{n+1-a}\wedge\cdots \wedge \a_{n}
$$
Then $X\in \wedge^kV$ maps to
$$
\sum_{|I|=k-a, a\leq i_1< \cdots < i_{a-k}<n+2-a}
 [\a_I\ww \a_{n+1-a}\wedge\cdots \wedge \a_{n}(X)]v_0\wedge\cdots \wedge v_a\ww v_I 
$$
Take $v=v_n$ and $X=v_{n-k+1}\ww \cdots \ww v_n\in \tker\phi_v$, then
$$
X\mapsto v_0\wedge\cdots \wedge v_a\ww v_{n-k+1 }\ww \cdots \ww v_{n-a+1}
\not\in v_n\ww \La kV=\tim(\phi_v).
$$
This proves the claim.\qed 

\medskip
\subsection{The general equivariant case}\label{equivaring}

\begin{question}
 {Given} three $G$-modules $U,V,W$, and  $T\in (U\ot V\ot W)^G$,  when is
one of the three associated spaces  of constant, or simply,   bounded  rank? 
\end{question}

Another way two ask the same question  is:  Given $G$-modules $U,W$, for which submodules
$V\subset U\ot W$ is the corresponding space of bounded rank?

If one takes the Cartan component of $U\ot W$, that is, if $U$ has highest weight $\mu$ and $W$
highest weight $\nu$, the submodule of highest weight $\mu+\nu$, then the  resulting space is not of bounded rank.
Indeed, by the Borel-Weil theorem we can interpret our three representations as spaces of
global sections of certain line bundles on the complete flag variety $G/B$, and our morphism
is given by the pointwise product of such sections; in particular, it is always injective. 
Examples indicate that the  {submodules} of lowest highest weight in $U\ot W$ are good  {candidates.}

\section{General linear group} \label{glsect}
\subsection{General remarks about embeddings $V\ra \thom(S_\mu V,S_\nu V)$.}
The irreducible representations of $GL(V)$ are the Schur modules $S_\mu V$, where $\mu$ can be 
supposed (after twisting by some character if necessary) to be a partition $\mu=(\mu_1,\ldots, \mu_\bv)$, 
with $\mu_1\ge \cdots \ge \mu_\bv\ge 0$ (in fact one may also assume $\mu_\bv=0$ but it
will be convenient not to impose this). 
 
By the Pieri  rule, $S_\nu V$ is contained in $S_\mu V\otimes V$  if and only if   $\nu=(\mu_1,\ldots,\mu_k+1, \ldots, \mu_\bv)$
for some integer $k$ such that $\mu_{k-1}>\mu_k$. 
The diagram of $\nu$ is then obtained 
by adding one box $b$ to the diagram of  {$\mu$}  at the extremity of the $k$-th row, as in the diagram below. 
For future use we denote the box immediately north of $b$  by $c$, if there is one. 

 \[   \young(~~~~~~,~~c~,~~b,~~,~)\]

Again by the  Pieri  rule, there is a unique  (up to scale) 
equivariant morphism
$$\phi : V \lra Hom(S_\mu V, S_\nu V).$$
Once we fix a nonzero vector $v\in V$, or the line $\ell=\CC v\subset V$, the image and the kernel of 
$\phi_v$ are preserved by the action of the stabilizer of $\ell$ in $GL(V)$, which is a parabolic subgroup $P$.
So we need to describe the $P$-module structure of $ S_\mu V$ and $ S_\nu V$.  First,   consider the action of the unipotent radical $P_u\subset P$, which is the subgroup 
acting trivially both on $\ell$ and on $V/\ell$; its (abelian) Lie algebra is   $\fp_u=Hom(V/\ell,\ell)
\subset End(V)$.
The action of $\fp_u$ on $S_\mu V$ has a nontrivial kernel $M_1$, and by induction one obtains  a canonical
filtration 
$$0=M_0\subset M_1\subset M_2\subset \cdots \subset M_m=S_\mu V,$$
such that $\fp_u(M_k)\subset M_{k-1}$. Consequently, the action of $P_u$ on each quotient $M_k/M_{k-1}$ is trivial, 
and the $P$-module structure of such a quotient is fully determined by its $L$-module structure, for $L$
a Levi factor of $P$. Concretely, fixing $L$ amounts to choosing a hyperplane $H$ in $V$ complementing
$\ell$, and then $L=GL(H)\times GL(\ell).$ As an $L$-module, $V=H\oplus \ell$ and the filtration of $S_\mu V$ 
that we have defined has associated grading determined by the $\ell$ degree in the
summands of   $S_\mu(H\oplus \ell)$. 

The decomposition of the latter $L$-module   may be described as follows:
\be\label{mudecomp}
S_\mu(H\oplus\ell)=\bigoplus_{k\ge 0}\bigoplus_{\mu\stackrel{k}{\ra}\alpha} S_{\alpha}H\otimes \ell^k,
\ene
where the symbol $\mu\stackrel{k}{\ra}\alpha$ means that the diagram of $\alpha$ can be obtained by deleting 
$k$ boxes from the diagram of $\mu$, deleting at most one box  per column.

 \[   \young(~~~~**,~~~*,~~,~*,~)\]
  In particular, it is important to notice
that this decomposition has no multiplicities bigger than one. 

There is a similar decomposition for $S_\nu V$: 
$$S_\nu(H\oplus\ell)=S_\nu H\oplus\bigoplus_{k\ge 0}\bigoplus_{\mu\stackrel{k+1}{\ra}\beta} S_{\beta}H\otimes \ell^{k+1}.$$
The morphism $\phi_v$ is just multiplication by $v\in\ell$, so it sends a component 
$S_{\alpha}H\otimes \ell^k$ of $S_\mu V$ to $S_{\alpha}H\otimes \ell^{k+1}$ if this is a component 
of $S_\nu V$, and to zero otherwise. In particular the kernel of $\phi_v$ is the direct sum of the 
components $S_{\alpha}H\otimes \ell^k$ of $S_\mu V$ such that  $S_{\alpha}H\otimes \ell^{k+1}$ 
does not appear in $S_\nu V$. 

When does this happen? Recall that $\alpha$ is obtained by erasing $k$ boxes from $\mu$, at most one per column. Moreover, $\nu=\mu\cup \{b\}$ for one box $b$. So to obtain $\alpha$ from $\nu$, one needs to erase two boxes from $\mu$ in the 
same column exactly when the box $c$ immediately north to $b$ does not belong to $\alpha$; we need to erase 
both $b$ and $c$. 

Similarly, the cokernel of $\phi_v$ is, as an $L$-module, the sum of the factors $S_{\beta}H\otimes \ell^{k}$ such that the box $b$ belongs to $\beta$. In summary: 

\begin{prop}\label{kerimcoker}
As $L$-modules, the kernel, image and cokernel of $\phi_v$ are:
$$\begin{array}{rcl}
Ker(\phi_v) & = & \bigoplus_{k\ge 0}\bigoplus_{\mu\stackrel{k}{\ra}\alpha, c\notin\alpha} S_{\alpha}H\otimes \ell^k,\\
Im(\phi_v) & = & \bigoplus_{k\ge 0}\bigoplus_{\mu\stackrel{k}{\ra}\alpha, c\in\alpha} S_{\alpha}H\otimes \ell^{k+1},\\
Coker(\phi_v) & = & \bigoplus_{k\ge 0}\bigoplus_{\nu\stackrel{k}{\ra}\beta, b\in\beta} S_{\beta}H\otimes \ell^{k}.
\end{array}$$
\end{prop}

\begin{coro}\label{injcor}
The morphism $\phi_v$ is injective (resp. surjective) exactly when the box $b$ belongs to the first row
(resp. the $\bv$-th row). 
\end{coro}


\subsection{
  Example:
  $\mu=(2)$ and $\nu=(2,1)$}\label{mu2nu21}

Given a decomposition $V=H\oplus \ell$ as above, we get
$$\begin{array}{ccl}
S_{2}V & = &  S_{2}H \quad \oplus  H\otimes \ell \quad \oplus  \ell^2, \\
S_{21}V & = & S_{21}H \quad \oplus  S^2 H\otimes \ell \quad \oplus  \La 2 H\otimes \ell\quad \oplus  H\otimes \ell^2.
\end{array}$$
Thus the kernel of $\psi$ is isomorphic to $\ell^2$, the image to $S_2H\otimes\ell \oplus H\otimes\ell^2$ and the 
cokernel to $S_{21}H\oplus \wedge^2H\otimes\ell$. 
This gives a matrix $\psi$ of linear forms in $n+1$ variables, of size $a_n\times b_n$ and constant rank $r_n$, where
$$a_n=\frac{(n+2)(n+1)}{2} ,\quad b_n=\frac{ n(n+1)(n+2)}{3}, \quad r_n=\frac{n^2+3n}{2}.$$

\begin{prop}
$\psi(V)$ is not rank-critical. 
\end{prop}

\proof First observe that Draisma's criterion does not apply by computing $RND(L)=L$ for $L=\psi(V)$. 
Recall that $RND(L)$ must be a submodule of
$$Hom(S_2V,S_{21}V)=S_{2,1,0^{n-2},-2}V\oplus S_{1,1,0^{n-2},-1}V\oplus S_{2,0^{n-1},-1}V\oplus V.$$
It suffices to consider their highest weight vectors, and the result of a straightforward 
computation is that $RND(L)=L\oplus S_{2,0^{n-1},-1}V$. This 
 suggests   considering $M=\langle L,\sigma\rangle$ 
for $\sigma$ a highest weight vector in $S_{2,0^{n-1},-1}V$. This is a tensor of the form $\sigma=e^2\otimes \alpha$
for $\alpha\in V^\vee$ a linear form vanishing on $e$. It sends $v^2$ to $\alpha(v)(e\wedge v)\otimes e=\alpha(v)\psi(v)$ $(e^2)$  which belongs to  {$Im(\psi(v))$}, so $\sigma$ belongs to $RND(L)$.  {We claim  } that $\psi(v)+\sigma$ is 
never injective. Indeed,  $\alpha(v)e^2-v^2$ is contained in its kernel. \qed 

\medskip It is plausible  {that} $\psi(V)$ is constant-rank-critical, in the sense
 that a bigger space of matrices will never have constant rank. Indeed having constant rank is not a closed condition, contrary to having bounded rank, so this does not follow from the previous discussion. 
  {For example, over $V\op \langle \s\rangle$, when 
 $n=2$ the map at $\s$ has rank $2<5$.}

\medskip\noindent {\it Remark.} 
When $n=2$ one obtains a three  dimensional space of $6\times 8$ matrices of constant
rank $5$.  Since   $4=8-5+1$ does not divide $5=5!/4!$, the maximum possible dimension of  such a space
is four  {\cite{MR878293}. }

\medskip
More  {intrinsically,}  the kernel bundle is  the  homogeneous line  bundle $\cO(-2)$, 
so in particular $c_1(\cE)=2$. 
The image bundle $\cE$ and the cokernel bundle $\cC$ fit into  short exact sequences  
$$0\rightarrow \cQ(-1)\rightarrow \cE\rightarrow S^2\cQ\rightarrow 0,\qquad 
0\rightarrow \wedge^2\cQ\rightarrow \cC\rightarrow S_{21}\cQ(1)\rightarrow 0.
$$
The situation can be summarized in the following diagram:

$$\xymatrix{\cO(-2) \ar@{->}[rd] & & & & \cC \\
 &  \ul{S^2V}  \ar@{->}[rd]\ar@{->}[rr]^{\psi} & & \ul{S_{21}V}(1) \ar@{->}[ru] & \\
  & & \cE  \ar@{->}[ru] & & 
}$$

\smallskip
An easy application of Bott's theorem  {(see, e.g.,
\cite{weyman}), recalling that $Ext^1(\cE,\cF)=H^1(\cE^*\ot \cF)$,  } {shows:}

\begin{lemma}\label{lem8}
As a $\fgl_n$-module,  $Ext^1(S^2\cQ, \cQ(-1))\simeq \fgl_n$. 
\end{lemma}
 
  {Lemma \ref{lem8}} suggests that $\cE$ might be deformed by changing the extension class, but this is not the case:

\begin{prop}
$\cE$ is rigid. 
\end{prop}

\begin{proof} Suppose $\cE_t$  were   a small deformation of $\cE_0=\cE$. Since $h^q(\cE)=0$ for any $q>0$, 
by semi-continuity this also holds for $\cE_t$ for $t$ close to zero, and therefore $h^0(\cE_t)=h^0(\cE)=6$.
Moreover the evaluation morphism $\underline{H^0(\cE_t)}\ra\cE_t$ remains surjective since this is an open 
condition. The kernel is a line bundle and  must be $\cO(-2)$ since the first Chern class of $\cE_t$ must be constant. The dual exact sequence 
$$0\ra \cE_t^\vee \ra \underline{H^0(\cE_t^\vee)}\ra \cO(2)\ra 0$$
induces  { a}  morphism $H^0(\cE_t)\ra H^0(\cO(2))$ which must be an isomorphism for $t$  { sufficiently small} since 
it is for $t=0$. Once we have identified these two spaces, we get the exact sequence whose kernel is 
$\cE^\vee$, which is thus isomorphic with $\cE_t^\vee$. 
\end{proof}

Recall the {\it slope} of a coherent sheaf $\cF$ is $\mu(\cF):=c_1(\cF)/\trank(\cF)$, and
 that by definition a vector bundle $\cE$ is {\it stable}
  if for all proper coherent subsheaves $\cF\subset \cE$ one has $\mu(\cF)<\mu(\cE)$.

\begin{prop}
$\cE$ is stable. 
\end{prop}

\proof 
The main result of \cite{MR1104341} states that a homogeneous vector bundle $\cE$ is stable if and only
if for all homogeneous subbundles $\cF\subset \cE$ one has $\mu(\cF)<\mu(\cE)$. In our case, 
such a subbundle yields a morphism $\cF\ra S^2\cQ$ which by Schur's lemma must be zero or surjective. 
In the first case $\cF$ is a subbundle of $\cQ(-1)$, and again by Schur's lemma $\cF$ being nonzero 
must coincide with $\cQ(-1)$; it is then easy to check that $\mu(\cF)<\mu(\cE)$. In the second case the 
kernel of the surjective map  $\cF\ra S^2\cQ$
must be a proper homogeneous subbundle of $\cQ(-1)$, so for the same reasons it must be zero and 
$\cF\simeq S^2\cQ$ yields a splitting of $\cE$, a contradiction. \qed 

\medskip
This confirms the expectation that this example yields an isolated point (up to change of basis) 
in the variety parametrizing matrices of constant rank. 
A straightforward computation gives the following $8\times 6$-matrix
 {where
the rows are respectively labeled by
$e_0^2, e_1^2,e_2^2, 2e_0e_1,2e_0e_2,2e_1e_3$
and the columns labeled by
$e_0\ww e_1\ot e_0, e_0\ww e_2\ot e_0, 
e_0\ww e_1\ot e_1,e_0\ww e_2\ot e_2,e_1\ww e_2\ot e_1,e_1\ww e_2\ot e_2,e_0\ww e_1\ot e_2+e_0\ww e_2\ot e_1,
e_0\ww e_1\ot e_2-e_1\ww e_2\ot e_0
$:  
$$\begin{pmatrix}
-y & -z & 0 & 0  &0 &0 &0 &0   \\
0  &0 &x &0 &-z &0 & 0& 0 \\
0   &0 & 0 & x& 0& y& 0& 0  \\
x    &0 &-y &0 &0 &0   &-z &z \\
0     &x &0 & -z& 0&0   & 0& -y\\
0      &0 &0 &0   &y &-z &x &0  
\end{pmatrix}.$$
}
We summarize the information about the low rank example:

\begin{prop}\label{newr5} The inclusion $\BC^3\subset \thom(S_2\BC^3, S_{21}\BC^3)$ 
is a  space of constant rank five  $6\times 8$-matrices of dimension $3$ that is rank critical and strongly indecomposable.
 The associated
vector bundles have first Chern classes $c_1(\cE)=2$, $c_1(\cE^\vee(1))=3$.
The associated space has no  nontrivial deformations.
The associated tensor does not have minimal border rank.
\end{prop}

All the assertions have been proven except the last. To prove it, consider the Koszul flattening
\cite{MR3376667}
  $V^*\ot S_2V\ra \La 2 V^*\ot S_{21}V$. This is a $GL(V)$-module map. It
it must have rank at least $12$ because the tensor is concise, but neither of the two irreducible modules
in the source has dimension $12$, so it must be full rank and thus the border rank of the tensor
must be  at least $9$.

More generally, if $\mu=(a)$ and $\nu=(a,1)$, one obtains a space of constant corank one with $c_1(\cE)=a$.

\subsection{Example: 
  $\mu=(2,2)$ and $\nu=(2,2,1)$}

Decomposing $V=H\oplus \ell$ as above, we get
$$\begin{array}{ccl}
S_{22}V & = &  S_{22}H \quad \oplus  S_{21}H\otimes \ell \quad \oplus  S_{2}H\otimes \ell^2, \\
S_{221}V & = & S_{221}H \quad \oplus  (S_{22}H\oplus S_{211}H)\otimes \ell\quad \oplus  S_{21}H\otimes \ell^2.
\end{array}$$
Thus  the kernel of $\psi$ is isomorphic to $ S_{2}H\otimes \ell^2$, and the image to 
$$S_{22}H\ot \ell \oplus  S_{21}H\otimes \ell^2.
$$
This gives a matrix $\psi$ of linear forms in  {$n+1$} variables, of size $a_n\times b_n$ and constant rank $r_n$, where
$$
a_n=\frac{n (n+1)^2(n+2)}{12} ,\quad b_n=\frac{(n+2)(n+1)^2n(n-1)}{24}, 
\quad r_n=\frac{n(n^2-1)(n+4)}{12}.
$$

The first nontrivial case is   $n=3$, where the image of the map $V\rightarrow Hom(S_{21}V,S_{221}V)$ is a four dimensional
space of matrices of size $20\times 20$, of constant rank $14$.   

The kernel bundle is the irreducible  homogeneous bundle $\cK=S^{2}\cQ(-2)$,
so $c_1(\cE)=n-1$.
The image bundle $\cE$ fits into a short exact sequence  
$$0\rightarrow S_{21}\cQ(-1)\rightarrow \cE\rightarrow S_{22}\cQ\rightarrow 0.$$

 In order to study the deformations of $\cE$, we will need the following lemma:

\begin{lemma}\label{lem10}
The dual bundle $\cE^\vee$ is acyclic. 
\end{lemma}

\proof 

 Start with the exact sequence 
\be\label{ex11}
0\rightarrow S^{2}\cQ(-2)\rightarrow  {\ul{S_{22}V}} \rightarrow \cE\rightarrow 0.
\ene

Dualize \eqref{ex11} and apply  Borel-Weil to see $H^0(S^{2}\cQ^\vee(2))=S_{22}V^\vee$. \qedsymbol

\begin{prop}
The vector bundle $\cE$ is   stable. It is not infinitesimally rigid, although it is rigid in the category of homogeneous bundles. Moreover $h^q(End(\cE))=0$ for $q>1$, so its deformations are unobstructed.
\end{prop}

\begin{proof} Twist \eqref{ex11} by $\cE^\vee$ to get 
$$0\rightarrow S^{2}\cQ(-2)\otimes\cE^\vee\rightarrow S_{22}V\otimes\cE^\vee\rightarrow End(\cE)\rightarrow 0.
$$
Since the central term  is acyclic by   Lemma \ref{lem10}, we deduce that 
$$H^q(End(\cE))\simeq H^{q+1}(S^{2}\cQ(-2)\otimes\cE^\vee).
$$
In order to compute this, consider exact sequence 
$$0\rightarrow S^{2}\cQ(-2)\otimes\cE^\vee\rightarrow S^{2}\cQ(-2)\otimes S_{22}V^\vee \rightarrow 
End(S^2\cQ)\rightarrow 0.
$$
By Bott's theorem $S^{2}\cQ(-2)$ is acyclic,  {so}  
$$H^q(End(\cE))\simeq H^q(End(S^2\cQ)).$$
Moreover, $End(S^2\cQ)$ has three irreducible components, namely 
$\cO$, $S_{10\ldots 0-1}\cQ$ and $S_{20\ldots 0-2}\cQ$. By Bott's theorem again, the second component is 
acyclic. The last one has a nontrivial cohomology group in degree one, 
namely the module  $S_{20\ldots 0-1-1}V$, that is 
$$H^1(End(\cE))=Ker(S^2V\otimes \wedge^2V^\vee\rightarrow 
V\otimes V^\vee).$$
In particular $H^1(End(\cE))^{SL(V)}=0$, which means that $\cE$
is infinitesimally rigid in the category of homogeneous bundles.
This concludes the proof. 
\end{proof}

\begin{prop}\label{moduli}
The vector bundle $\cE$ has
an $ {\frac{(n+1)^2(n^2+2n-4)}4+1}$-dimensional  
space of  deformations $\tilde \cE$ which are not homogeneous bundles, but keep the 
property that $\tilde\cE$ and $\tilde \cE^\vee(1)$ are generated by global sections. 
\end{prop}

\proof Using the exact sequence 
$$0\rightarrow S_{21}\cQ(-1)\rightarrow \cE\rightarrow S_{22}\cQ\rightarrow 0$$
and  {Bott's}  theorem, it is straightforward to check that $\cE$ and $\cE^\vee(1)$
have no higher cohomology. By semi-continuity this must remain true for a small deformation $\tilde \cE$, 
and consequently $h^0(\tilde \cE)$ remains constant under small deformations. In this situation
the condition to be generated by global sections is open, and the claim follows. \qed 

\medskip 
The   unexpected consequence, at least for us, is that in this case it is possible to
deform $\psi$ into a non-equivariant morphism of vector spaces, while keeping the
property that the rank is constant.

\subsection{ Example:  $\mu=(2^a,1^b)$ and $\nu=(2^a,1^{b+1})$}

Here  the kernel $\cK$ of $\phi$, its image $\cE$ and its cokernel $\cC$ fit into
simple exact sequences, namely
$$0\rightarrow S_{2^{a-1},1^{b}}Q(-2)\rightarrow \cK\rightarrow S_{2^{a },1^{b-1}}Q(-1)\rightarrow 0,$$
$$0\rightarrow S_{2^{a-1},1^{b+1}}Q(-1)\rightarrow \cE\rightarrow S_{2^{a },1^{b}}Q\rightarrow 0,$$
$$0\rightarrow S_{2^{a-1},1^{b+2}}Q\rightarrow \cC\rightarrow S_{2^a,1^{b+1}}Q(1)\rightarrow 0.$$
We   deduce  
$$\begin{array}{ll}
H^0(\cE)=S_{2^a,1^b} V, & H^0(\cE^\vee(1))=S_{2^a,1^{b+1}} V^\vee, \\ 
H^0(\cK(1))=S_{[a,b-1]} V, & H^0(\cK^\vee)=S_\mu V^\vee, \\ 
H^0(\cC(-1))=S_\nu V, & H^0(\cC^\vee(2))=S_{[a,b+2]} V^\vee. 
\end{array}$$
In particular all these homogeneous bundles must be indecomposable, and even stable by \cite{MR1104341}.
We expect that,  
  as in the previous example, they can be deformed to non-homogeneous bundles. 

Consider the special case $a=1$, which is closest to the classical case $\La k V\ra \La{k+1}V$.
Here one has rank
$r=\binom{n}{b+2}+\binom nb\frac{(n-b)(n-1)}{b+2}$.
The first case is $n=3$, $b=1$ where one has a four dimensional subspace
of $\BC^{20}\ot \BC^{15}$ of rank $11$. In this case $c_1(\cE)=7$, so $c_1(\cE^\vee(1))=4$. (When $a=b=1$, $c_1(\cE)=(n^2+3n-4)/2$.)

\medskip\noindent 


\subsection{$GL(V)$-equivariant spaces of bounded rank with base not $\BP V$}
When the base space is not given by the natural representation of $GL(V)$, it  {is} more difficult 
to construct equivariant spaces of matrices of constant,  or 
 {even} bounded rank. A well-known example is provided by 
the adjoint action of $\fsl(V)$ on itself; the general element being regular semisimple, its commutator is 
a Cartan subalgebra and the generic corank of the adjoint action is thus  $n$; adjoint actions of simple Lie algebras are discussed in \cite{MR2268360} and proved to be rank critical. 

In this section we give  {additional}
  examples of spaces of bounded rank, and ask if their main common feature may explain their existence. 

\begin{example}\label{wedge}
Let $V=\Lambda^2A$, and let $\aaa=2p$ be even.  Then $S_{2^p}A$ is a submodule of $S^2(\La p A)$ while 
$S_{2^p11}A$  is a submodule of $\La 2 (\La {p+1} A)$, both with multiplicity one.
 {Consider    
\begin{align*}
 \phi: \La 2 A&\hookrightarrow\thom(S_{2^p}A, S_{2^p11}A)\\ 
  v\ww w & \mapsto \Big(
X^2
\mapsto
(X\ww v)\ot  ( X\ww w)- (X\ww w)\ot  ( X\ww v)\Big),
\end{align*}}
  and extending linearly.

If $z=v_1\ww w_{1}+\cdots +v_p\ww w_{p}\in \Lambda^2A$   is a general element, 
 {then} $X^2\in \tker \phi_z$ when $X=v_1\ww v_2\ww\cdots \ww v_p$. Such an $X$ is a Pl\"ucker
representative of a $p$-dimensional subspace $V$ of $A$ which is isotropic with respect to 
the skew-symmetric two-form $\omega$ that is dual to $z$. Conversely, such a $V$ being given, 
one can choose another isotropic subspace $W$ of $A$ which is transverse to $V$, in which case 
$\omega$ restricts to a perfect duality between $V$ and $W$. If $(v_1,\ldots ,v_p)$ and 
$(w_1,\ldots ,w_p)$ are dual basis of $V$ and $W$, then we can write $z=v_1\ww w_{1}+\cdots +v_p\ww w_{p}$, 
and $v_1\ww v_2\ww\cdots \ww v_p$ is  a Pl\"ucker representative for $V$. 

We conclude that the linear span  {$K\isom V_{2\om_p}^{Sp(2p)}$} of the second Veronese image of the Lagrangian Grassmannian $ {\langle v_2(LG_z(p,2p))\rangle}$ is contained in the kernel of $\phi_z$. This is unexpected since 
$$\tdim S_{2^p}A=\frac 1{p+1}\binom \aaa p\binom{\aaa+1}p
<\tdim S_{2^p11}A=\frac 3{\aaa-p+1}\binom \aaa p\binom{\aaa+1}{p+3}$$
for $p\ge 5$. The dimension of $K$ can be computed from the Weyl dimension formula for the symplectic group 
$Sp(2p)$, which gives 
$$\dim K =24\frac{(2p+1)!(2p+3)!}{p!(p+1)!(p+3)!(p+4)!}. $$
For $p=5$ this gives matrix of size $19404\times 20790$ of rank bounded by $14685$.
\end{example}

\begin{example}
Let $V=\Lambda^2A$, and let $\aaa=2p+1$ be odd.  
Consider an equivariant embedding $\La 2 A\subset\thom(S_\lambda A, S_\mu A )$. That such an embedding 
exists is equivalent to the condition that $\mu$ can be obtained by adding two boxes to $\lambda$,
not on the same row. The embedding is  { unique  up to scale}.

A general element $z=v_1\ww w_{1}+\cdots +v_p\ww w_{p}\in \Lambda^2A$ defines a unique hyperplane $H_z$ of $A$ such that $z$ belongs to $\wedge^2H_z$. The equivariance of $\phi$ implies that there is a commutative diagram
$$\xymatrix{ S_\lambda A\ar@{->}[rr]^{\phi_z}  & & S_\mu A \\
 S_\lambda H_z\ar@{->}[rr]^{\phi_z^{H_z}}\ar@{(->}[u] & &  S_\mu H_z\ar@{->}[u]}$$ 
 {Set $s_\lambda(n)=\tdim(S_\lambda\CC^n)$.}  
 {If 
$$s_{\lambda}(2p)-s_{\mu}(2p)> s_{\lambda}(2p+1)-s_{\mu}(2p+1),$$
then $\phi$ has bounded rank because
  the factorization above shows that the kernel of $\phi_z$    contains the kernel 
of  $\phi_z^{H_z}$.}

Take for example $p=2$, $\lambda = (32)$ and $\mu =(3211)$. Then $  s_{\lambda}(5)=s_{\mu}(5)=175$.  {
Here  $s_{\lambda}(4)=60> s_{\mu}(4)=20$,}
so $\phi$ has at least a $40$-dimensional kernel.
\end{example}
  
\begin{example}\label{wedge3C7}
Let $A$ be seven dimensional and let $V=\wedge^3A$. The action of $SL(A)$ on $V$ yields a morphism
$\fsl(A)\otimes V\lra V$ that we can see as a map $V\lra Hom(\fsl(A),V)$. Since $\dim\fsl(A)=48
>\dim V=35$, one could expect the generic point in the image of this  morphism to be surjective 
but we claim this is not the case. 
Indeed, it is enough to check it at a general point $\omega\in V$. As already known to E. Cartan, 
the stabilizer of $\omega$ in $SL(A)$ is then a copy of $G_2$, and the associated morphism $\phi_\omega$ in 
$Hom(\fsl(A),V)$ is $\fg_2$-equivariant.  {Compare the decompositions of $ \fsl(A)$ and $V$
into $\fg_2$-modules:} 
$$ \fsl(A)=\fg_2\oplus V_{2\omega_1}\oplus A, \qquad V= V_{2\omega_1}\oplus A\oplus \CC,$$
where $V_{2\omega_1}$ is a hyperplane in $S^2A$. The important point in these decomposition is that the 
trivial factor $\CC$ appears on the right hand side but not on the left hand side. Therefore by Schur's 
Lemma it cannot be contained in the image of $\phi_\omega$. We thus get a $35$-dimensional space $\phi(V^\vee)$ 
of $48\times 35$-matrices of generic rank $34$.


 {A similar}  argument applies when $A$ is eight dimensional and  $V=\wedge^3A$. Since $\dim\fsl(A)=63
>\dim V=56$, one could again expect the generic point in the image of the  morphism $V\lra Hom(\fsl(A),V)$
to be surjective, and again this is not the case. To see this, it suffices to observe that the stabilizer
in $SL(A)$ of a generic element $\omega$ in $V$ is a copy of $SL(3)$. Then as before $V$ contains a trivial 
$SL(3)$-module (generated by $\omega$) while $\fsl(A)$ does not, and therefore $\phi_\omega$ cannot be surjective. 

Note finally that  {when  $\tdim A>8$}, the dimension of $\fsl(A)$ gets smaller than the dimension
of $V=\wedge^3A$, so the argument  {no longer applies}. 
\end{example}

These are all the examples we are aware of. Hence the following:

\medskip\noindent {\bf Question}. {\it Suppose that  $V=S_\pi A$
admits a $GL(A)$-embedding  as a space of bounded rank. Is the stabilizer in $GL(A)$ of 
a general element of $V$ positive dimensional?}

\section{Big matrices from small ones}

\subsection{A motivating example}
Consider a matrix $M$ of size $a\times (a+b)$ of linear forms, with constant rank $a<a+b$. This will be the
case of a general matrix whose entries are general combinations of $c+1$ linear forms, with $c\le b$. 
In this case, there is an Eagon-Northcott complex
\cite{MR142592} of sheaves on $\PP^c$ which is everywhere exact, and 
can be written as
$$0\ra S^{b}A\otimes\cO(-b)\ra  S^{b-1}A\otimes B\otimes\cO(-b+1)\ra \cdots \hspace*{2cm}$$ 
$$\hspace*{3cm} \cdots 
\ra  A\otimes\wedge^{b-1}B\otimes \cO(-1)\ra  \wedge^{b}B\otimes \cO\ra 0.$$ 
Here $A$ has dimension $a$, $B$ has dimension $a+b$ and $M$ is interpreted as a element in $Hom(A,B)$
whose entries are linear forms. The morphisms in the complex, up to twist, are 
$$ e_k(M): S^{k+1}A\otimes \wedge^{b-k-1}B\otimes\cO(-1)\ra  
S^{k}A\otimes \wedge^{b-k}B\otimes\cO$$ 
and can easily be defined by contraction with $M$. Since the Eagon-Norton complex is in this case an
exact complex of vector bundles, the rank of $e_k(M)$ must be constant. In other words we obtain for each 
$M$ and each $k$, a matrix $e_k(M)$  which is of constant rank. 

This observation can be widely generalized. 

\subsection{How to build big matrices}
Consider partitions $\lambda, \lambda'$ and $\mu, \mu'$ such that $\lambda'$ is obtained by suppressing 
one box of $\lambda$, and $\mu'$ is obtained by adding one box to $\mu$. There is then a  {unique up to scale}
equivariant morphism 
$$\Theta : Hom(A,B)\otimes S_\lambda A\otimes S_\mu B\lra S_{\lambda'} A\otimes S_{\mu'} B. $$
So any morphism $X\in  Hom(A,B)$ induces a morphism $\Theta_X : S_\lambda A\otimes S_\mu B\lra S_{\lambda'} A\otimes S_{\mu'} B. $

\begin{prop}\label{constant-gen}
The rank of $\Theta_X $ only depends on the rank of $X$. As a consequence, if $M\subset Hom(A,B)$ is a space
of morphisms of constant rank, then $\Theta_M$ also is.
\end{prop}

\proof Since the construction is equivariant, the rank of $\Theta_X$  {is constant on the $GL(A)\times GL(B)$-orbits,
which are indexed by the matrix rank.} \qed

\medskip  {The} same argument  {yields} a more general version with $GL(A)\times GL(B)$-modules that
are not  {necessarily} irreducible. This is just an expansion of Proposition \ref{constant}, but the  {greater} 
generality allows  {one} to construct infinitely many spaces of matrices of constant rank just from  one. 

\subsection{A simple example of Eagon-Northcott type}
Consider one of the simplest morphisms appearing in an Eagon-Northcott complex, namely:
$$ \Theta_X : S^2 A\otimes B\lra A\otimes \wedge^2B $$
defined by $X\in Hom(A,B)$. 

\begin{prop}
Suppose that $X$ has rank $r$, then $$ \rank \Theta_X =
abr-a\binom{r+1}{2}-b\binom{r}{2}+2\binom{r+1}{3}.
$$

\end{prop}

\begin{proof} Choose splittings $A=\tker(X)\op A'$
and $B=B'\op B''$, where $B'=\tim(X)\isom A'$. Then $\Theta_X$ decomposes as a sum
of the following maps, whose various images we examine separately:
\begin{align*}
S^2\tker X\ot B&\ra 0\\
(\tker X)\cdot A'\ot B'&\ra \tker X\ot \La 2 B'\\
(\tker X)\cdot A'\ot B''&\ra \tker X\ot   B'\ww B'' \\
S^2A'\ot B''&\ra A'\ot B'\ww B''\\
S^2A'\ot B'&\ra A'\ot \La 2 B'
\end{align*}
This yields a block decomposition of $\Theta_X$, whose rank is therefore the sum
of the ranks of the maps above. Those ranks are easy to compute: the first map is zero,
the second one is surjective, the third one is an isomorphism, the fourth one is injective; 
the last one can be identified to the morphism $S^2A'\ot A'\ra A'\ot \La 2 A'$,
whose image is $S_{21}A'$.
Thus the rank of $\Theta_X$ is
$$
(a-r)\binom r2+(a-r)r(b-r)+\binom{r+1}{2} (b-r)+\frac{r^3-r}3,
$$
which after a slight rewriting yields our claim.\end{proof}

\section{Symplectic group}\label{spsect}
\subsection{General set-up}
Branching rules for restrictions of irreducible representations from $Sp(2n)$ to $Sp(2p)\times Sp(2q)$,
where $n=p+q$, are given in \cite{MR411400}. Irreducible representations of $Sp(2n)$
are indexed by partitions $\lambda$ of length at most $n$; we denote them by $S_{\langle \lambda\rangle}W$
where $W$ is the natural representation of dimension $2n$. The representation $S_{\langle \lambda\rangle}W$
is the submodule of  $S_\lambda W$ consisting of  the common kernel of all the possible contractions by the symplectic form. 
Specializing formula \cite[(4.15)]{MR411400} to $p=n-1$ and $q=1$ we obtain:
\be\label{sympdecomp}
S_{\langle \lambda\rangle}(\BC^{2p}\oplus\CC^2)=\bigoplus_{\ell(\zeta)\le 2} 
S_{\langle \lambda/\zeta\rangle}\BC^{2p}\otimes S_\zeta\CC^2,
\ene
where $S_{\langle \lambda/\zeta\rangle}U=\bigoplus_{\eta}c_{\zeta\eta}^\lambda S_{\langle \eta\rangle}U$
and the $c_{\zeta\eta}^\lambda$ are the Littlewood-Richardson coefficients. 

The stabilizer $P$ in $Sp(W)$ of a line $\ell\in\PP(W)$ is a parabolic subgroup that  also preserves the 
hyperplane $\ell^\perp$. Its unipotent radical $P_u$ is the subgroup that acts trivially on the three factors 
$\ell, \ell^\perp/\ell, W/\ell^\perp$. A Levi factor $L$ is obtained by  {choosing} a decomposition $W=\ell\oplus\ell'\oplus H$,
where $H$ is the orthogonal complement (with respect to the symplectic form)  to $\ell\oplus\ell'$, so that $ \ell^\perp=\ell\oplus H$. In particular \eqref{sympdecomp}
yields the decomposition of $S_\lambda W$ as an $L$-module. 

It is then straightforward to find a criterion ensuring that the natural morphism 
$\phi: W\lra Hom(S_{\langle \mu\rangle}W,S_{\langle \nu\rangle}W)$ has bounded rank,  for two 
partitions $\mu, \nu$ such that $\nu/\mu$ is one box $b$. 
 \[   \young(xxxxxxx,xxxxxx,xxc~,xxb,xx,x)\]

\begin{prop}
The morphism $\phi_v$ is never surjective. It is not injective as long as the box $b$ does not belong to 
the first  row. 
\end{prop}

\proof Let $\ell=\CC v$, and fix an orthogonal decomposition  $W=H\oplus (\ell\oplus \ell')$ defining   
  a Levi factor $L$ of the parabolic subgroup $P\subset Sp(W)$ that fixes $\ell$. Let $K=\ell\oplus \ell'$. 
As  $L$-modules,   
$$
S_{\langle \mu\rangle}(H\oplus K)=\bigoplus_{\ell(\zeta)\le 2} S_{\langle \mu/\zeta\rangle}H\otimes S_\zeta K, 
$$
$$
S_{\langle \nu\rangle}(H\oplus K)=\bigoplus_{\ell(\delta)\le 2} S_{\langle \nu/\delta\rangle}H\otimes S_\delta K.
$$
The multiplication by $v$ acts on the $K$-factors and leaves the $H$-factors untouched. All the partitions 
appearing in $\mu/\zeta$ are contained in $\mu$, in particular $\nu$ is not one of them, while it appears in 
$ \nu/\delta$ for the empty partition  $\delta$. This proves that $\phi_v$ cannot be surjective. 

Conversely, to prove that $\phi_v$ is not injective, we make the following observation. 
Let $m>0$ be the number of boxes of $\mu$ to  the east of $c$, including it. Let $\theta$ be the 
partition obtained by suppressing these boxes from the diagram of $\mu$. Since these boxes all 
belong to the same row, $\theta$ appears in some $\mu/\zeta$ if and only if $\zeta=(m)$. 
Similarly, the Littlewood-Richardson rule implies that $\theta$ appears in some $\nu/\delta$
if and only if $\delta=(m,1)$. So $\phi_v$ has to send $S_{\langle \theta\rangle}H\otimes S^mK$
to $S_{\langle \theta\rangle}H\otimes S_{m,1}K$. But $S_{\langle \theta\rangle}H\otimes v^m $
is in the kernel of this map. \qed



\subsection{Example: $\BC^6\subset Hom(\wedge^{\langle 2\rangle}\BC^6,\wedge^{\langle 3\rangle}\BC^6)$}
 Consider
$G=Sp_6$ and let $W$ denote the natural six-dimensional representation. Then the image of the 
map $W\rightarrow Hom(\wedge^{\langle 2\rangle}W,\wedge^{\langle 3\rangle}W)$ is a six dimensional
space of matrices of size $14\times 14$, of constant rank $9$. 

To be more explicit, a line $\ell$ in $W$ determines a hyperplane $H=\ell^\perp$,  
 and   in $\wedge^2W$ one obtains two flags as follows  
$$\xymatrix{
 & & \wedge^2H\ar@{->}[rd] & & \\
 0  \ar@{->}[r] & \ell \wedge H \ar@{->}[rd]\ar@{->}[ru] & & \ell \wedge W+\wedge^2H \ar@{->}[r] & \wedge^2W\\ 
 & & \ell \wedge W\ar@{->}[ur] & &
}$$
There are two different ways to get a subspace in $\wedge^{\langle 2\rangle}W$ from a subspace $U$ of  $\wedge^2W$, 
either by taking the intersection with $\wedge^{\langle 2\rangle}W$, or by considering the projection $\overline{U}$ 
according to the decomposition $\wedge^2W=\wedge^{\langle 2\rangle}W \oplus\CC\omega$, where $\omega$ denotes the (dual) symplectic form. We  {obtain} a diagram of the same shape:
$$\xymatrix{
 & & \wedge^{\langle 2\rangle}H\ar@{->}[rd] & & \\
 0  \ar@{->}[r] & \overline{\ell \wedge H} \ar@{->}[rd]\ar@{->}[ru] & & \overline{\wedge^2H} \ar@{->}[r] & 
 \wedge^{\langle 2\rangle}W\\ 
 & & \overline{\ell \wedge W}\ar@{->}[ur] & &
}$$
 {It is straightforward to compute}  the kernel and image of $\phi$ in this case,  {which} are 
respectively the five-dimensional space $\overline{\ell \wedge W}$ and the nine-dimensional space
$\ell \wedge  \wedge^{\langle 2\rangle}W$. 

In terms of vector bundles, we deduce the following result. 
The quotient $\cN=\cL^\perp/\cL\simeq\cN^\vee$ is the {\it{null-correlation bundle}} (see, e.g., \cite[I.4.2]{oss}). 
This is a rank four bundle with an invariant symplectic form, induced by the one on $W$.
It satisfies $c(\cN)=1+h^2+h^4+\cdots + h^{n-1}$.

\begin{lemma} 
The  image $\cE$ of $\phi$ is a homogeneous bundle of rank nine,
fitting into an extension 
$$ 0\rightarrow\wedge^{\langle 2\rangle}\cN\rightarrow\cE\rightarrow
\cN(1) \rightarrow 0.$$ 
\end{lemma} 

By the computations in the proof that follows, $Ext^1(\cN(1), \wedge^{\langle 2\rangle}\cN)=\CC$,
so this extension is unique. In fact there is a commutative diagram 

$$\begin{CD}
  @. 0   @. 0   @.  @. \\
    @. @VVV        @VVV   @. \\
  @. \cO   @= \cO   @.  @. \\
    @. @VVV        @VVV   @. \\
  0  @>>> \wedge^2\cN   @>>> \wedge^2\cQ   @>>> \cN(1)   @>>> 0 \\
    @. @VVV        @VVV   @| \\
  0  @>>> \wedge^{\langle 2\rangle}\cN   @>>> \cE  @>>> \cN(1)   @>>> 0 \\
     @. @VVV        @VVV   @. \\
      @. 0   @. 0   @.  @. 
\end{CD}$$

\medskip\noindent
and the middle vertical exact sequence   defines $\cE$ as the quotient of $\wedge^2\cQ$ by its
global section defined by the (dual) symplectic form $\omega\in\wedge^2W$. 

In contrast to the previous examples, we have:

\begin{prop}
$\cE$ is   stable and not rigid. 
\end{prop}

\proof Tensoring the middle vertical sequence  with $\cE^\vee$ gives
$$ 
0\rightarrow\cE^\vee\rightarrow \cE^\vee\otimes\wedge^2\cQ\rightarrow End(\cE)\rightarrow 0.
$$
 {By} Bott's theorem $\wedge^2\cQ^\vee$ is acyclic, the only nonzero cohomology group of $\cE^\vee$ 
is $H^1(\cE^\vee)=\CC$. To compute the cohomology of the second term we use the dual of the middle
vertical sequence tensored with $\wedge^2Q$ to get
$$ 
0\rightarrow \cE^\vee\otimes\wedge^2\cQ\rightarrow \wedge^2\cQ^\vee\otimes\wedge^2\cQ\rightarrow  \wedge^2\cQ
 \rightarrow 0.
 $$ 
 The tensor product $\wedge^2\cQ\ot\wedge^2\cQ^\vee$ has three components, one of which is trivial and the 
 other two are acyclic. We  {conclude }  $\cE^\vee\otimes\wedge^2\cQ$ has only one nonzero cohomology group,
 namely $H^1(\cE^\vee\otimes\wedge^2\cQ)=\wedge^{\langle 2\rangle}W$. We  {deduce}
$$
H^0(End(\cE))=\CC, \quad H^1(End(\cE))=\wedge^{\langle 2\rangle}W, 
\quad H^q(End(\cE))=0 {\rm \ for \ } q>1.
$$
 So $\cE$ is simple but not infinitesimally rigid. The stability follows from \cite{MR1104341}.
\qed

\medskip
In this case the non-rigidity  {is} explained by the action of $SL(W)$, since our bundle $\cE$ is only
$Sp(W)$-homogeneous, but not $SL(W)$-homogeneous.  {By} varying the symplectic form  {one obtains} a family of bundles parametrized 
by $SL(W)/Sp(W)$, whose tangent space at the identity is precisely $\wedge^{\langle 2\rangle}W$.

\medskip
We   exhibit our $14\times 14$ matrix by chosing an adapted basis $e_1,\ldots ,e_6$ of $W$, 
in which the (dual) symplectic form is $\omega=e_1\wedge e_2+e_3\wedge e_4+e_5\wedge e_6$. We get
the following constant rank matrix $\psi_v$, depending on $v=(x_1,\ldots ,x_6)\in W$:
\small
$$\left(\begin{array}{cccccccccccccc}
 x_3 &  0 &  -x_2 &  0 &  0 &  x_1 &  0 &  0 &  0 &  -x_6 &  x_5 &  0 &  0 & 0  \\
 x_4 &  0 &  0 &  -x_2 &  0 &  0 &  0 & x_1 &  0 &  0 &  0 &  0 &  -x_6 & -x_5  \\
 x_5 &  -x_5 &  0 &  0 &  -x_2 &  0 &  0 &  0 &  x_1 &  0 &  x_4 &  0 &  -x_3 & 0  \\
 x_6 &  -x_6 &  0 &  0 &  0 &  -x_2 &  0 &  0 &  0 &  x_1 &  0 &  x_4 &  0 & -x_3  \\
 0 &  x_1 &  x_4 &  -x_3 &  -x_6 &  x_5 &  0 &  0 &  0 &  0 &  0 &  0 &  0 & 0  \\
 0 &  x_2 &  0 &  0 &  0 &  0 &  x_4 &  -x_3 &  -x_6 &  x_5 &  0 &  0 &  0 & 0  \\
 0 &  0 &  x_5 &  0 &  -x_3 &  0 &  0 &  0 &  0 &  0 &  x_1 &  0 &  0 & 0  \\
 0 &  0 &  x_6 &  0 &  0 &  -x_3 &  0 &  0 &  0 &  0 &  0 &  x_1 &  0 & 0  \\
 0 &  0 &  0 &  x_5 &  -x_4 &  0 &  0 &  0 &  0 &  0 &  0 &  0 &  x_1 & 0  \\
 0 &  0 &  0 &  x_6 &  0 &  -x_4 &  0 &  0 &  0 &  0 &  0 &  0 &  0 & x_1  \\
 0 &  0 &  0 &  0 &  0 &  0 &  x_5 &  0 &  -x_3 &  0 &  x_2 &  0 &  0 & 0  \\
 0 &  0 &  0 &  0 &  0 &  0 &  x_6 &  0 &  0 &  -x_3 &  0 &  x_2 &  0 & 0  \\
 0 &  0 &  0 &  0 &  0 &  0 &  0 &  x_5 &  -x_4 &  0 &  0 &  0 &  x_2 & 0  \\
 0 &  0 &  0 &  0 &  0 &  0 &  0 &  x_6 &  0 &  -x_4 &  0 &  0 &  0 & x_2  
 \end{array}\right)$$
\normalsize

\bigskip
This has a   curious consequence for  the Koszul map $$ \bar{\psi}: W\lra Hom(\wedge^2W, \wedge^3W),$$
which is of constant rank $10$. Once we have chosen a symplectic form on $W$, we get direct sum decompositions
$\wedge^2W=\wedge^{\langle 2\rangle}W\oplus\CC$, $\wedge^3W=\wedge^{\langle 3\rangle}W\oplus W$ which
induces a decomposition of
  $\bar{\psi}$ into blocks  as follows:
\small
$$\bar{\psi}_v=\begin{pmatrix}
\begin{matrix} & & & & & & & &  \\
 & & & & & &  & &  \\
 & & & & & & & &  \\
& & & & \psi_v& & & &   \\
 & & & &  && & & \\
 & & & & & & && \\
 & & & & & &  & &
 \end{matrix} & \vline & \begin{matrix} 0\\ 0\\0\\ 0\\ 0\\ 0\\ 0\\ 0 \end{matrix}
 \\
 \hline 
 \begin{matrix} & &&& & & & &  \\
 & & & & & &  && \\
 & & & & & & &&  \\
&& & &\theta_v & & &&  \\
 & & & & & &&& \\
 & & & & & & &&
 \end{matrix} & \vline 
 & \begin{matrix} x_1\\ x_2\\ x_3\\ x_4\\ x_5\\ x_6 \end{matrix}
  \end{pmatrix}$$
\normalsize
where $\theta^t$ is the matrix of the morphism $W\lra Hom(W,\wedge^{\langle 2\rangle}W)$. 
The string of zeroes is explained by the fact that there is no nonzero equivariant map 
$W\lra Hom(\CC,\wedge^{\langle 3\rangle}W)$. The rank
of $\bar{\psi}_v$ must be at least equal to the rank of $\psi_v$ plus one, and  {in fact} 
  there is always equality. 
In particular the space
$$\begin{pmatrix} \psi \\ \theta \end{pmatrix}: W\ra  \thom(  \wedge^{\langle 2\rangle}W, \wedge^{  3 }W)
$$
is of constant rank $9$ so $\psi$ is expandable. 

 {
\begin{question} Are all $SP(W)$ spaces with base $W$   expandable to $SL(W)$-spaces? If not, how to distinguish which are?
\end{question}
}

\section{Orthogonal groups: tensorial representations}\label{sosect}
In the case where $G=SO(W)$ is a special  orthogonal group   
Proposition \ref{constant} will in general
fail to hold, as one   expects the morphism $\phi$ to degenerate along the invariant 
quadric $Q\subset\PP(W)$. This is not always the case, obvious counter-examples arise from
  $SO(W)$-modules that  are restrictions of $SL(W)$-modules. We   discuss
two  non-obvious
  examples  where the base space is $\PP(W)$.

Irreducible representations of $SO(m)$ with 
support on the first $\lfloor m/2\rfloor -2$ fundamental weights
are indexed by partitions $\lambda$ of length at most 
$m/2-2$ when $m$ is even
and $(m-1)/2 -1$ when $m$ is odd; we denote them by $S_{[ \lambda]}W$
where $W$ is the natural representation of dimension $m$. As in the symplectic case, the representation 
$S_{[ \lambda]}W$ can be defined in $S_\lambda W$ as the common kernel of all the possible contractions
by the invariant quadratic form. Similarly, if the partition $\nu$ is obtained by adding a box $b$ to a partition 
$\mu$, there is a unique (up to scale) equivariant morphism 
$\phi: W\lra Hom(S_{[\mu]}W,S_{[\nu ]}W)$.

Once we fix a non-isotropic vector $v\in W$, we get an orthogonal decomposition $W=\ell\oplus\ell^\perp$, 
where $\ell=\CC v$. Moreover $\ell^\perp$ inherits an invariant quadratic form  {giving rise to}  a copy of 
$SO(m-1)$ inside $SO(m)$, that acts trivially on $\ell$. In particular the morphism $\phi_v$ is $SO(m-1)$-invariant. 
Formula \cite[(4.12)]{MR411400} gives the following decomposition:
\be\label{sobranch}
S_{[\mu]}(\ell\oplus\ell^\perp)=\bigoplus_{k\ge 0} S_{[\mu/k]}(\ell^\perp)\otimes \ell^k.
\ene
Formally this is the same decomposition as \eqref{mudecomp} that we  used for $SL(W)$, and we can just mimic  the proof of 
Proposition \ref{constant} to  {obtain}:

\begin{prop}
The morphism $\phi_v$ is never surjective. It is not injective as long as the box $b$ does not belong to 
the first  row. 
\end{prop}

\subsection{Case   $\mu=(2)$ and $\nu=(2,1)$} 
Here $S_{[2]}W$ is the hyperplane of $S_2W$ generated by the squares of the isotropic vectors; its invariant complement 
is generated by the dual $\hat{q}$ of the quadratic form. The natural composition $W\otimes \hat{q}\ra W\otimes 
S_2W \ra S_{2,1}W$ is an embedding, and the cokernel is a copy of $S_{[2,1]}W$. We  {may} therefore describe the map 
$\phi : W\ra Hom(S_{[2]}W, S_{[2,1]}W)$ in terms of $\phi : W\ra Hom(S_2W, S_{2,1}W)$ by sending $v\in W$ to the  
composition
$$\phi_v : S_{[2]}W\hookrightarrow S_2W\stackrel{\psi_v}{\lra}S_{2,1}W\lra S_{[2,1]}W.$$
In order to compute the kernel of $\phi_v$, 
we note that any $\kappa\in S_2W$ can be written as $\kappa=\sum_i \kappa_i e_i^2$ for some $q$-orthonormal basis $e_1,\ldots , e_m$ of $W$; it belongs to $S_{[2]}W$ when 
$\sum_i \kappa_i=0$. Since $\hat{q}=\sum_i e_i^2$, the kernel of the projection $S_{2,1}W\lra S_{[2,1]}W$
is the space of tensors of the form $\sum_i w\wedge e_i\otimes e_i$, for $w\in W$. So $\phi_v(\kappa)=0$
if and only if there exists $w\in W$ such that 
$$\sum_i \kappa_i v\wedge e_i\otimes e_i=\sum_i w\wedge e_i\otimes e_i,$$
which means that for each $i$, we have $w=\kappa_i v+\mu_i e_i$ for some scalar $\mu_i$.
If there exists two indices $i\ne j$ such that $\kappa_i\ne \kappa_j$, we deduce that 
$v$ and $w$ belong to $\langle e_i,e_j\rangle$. 
Then for $k\ne i,j$ we must have $\mu_k=0$ and $w=\kappa_k v$, hence $(\kappa_k-\kappa_i)v=\mu_i e_i$
and $(\kappa_k-\kappa_j)v=\mu_j e_j$. So necessarily, up to changing $i$ and $j$, $\mu_j=0$. 
Then we conclude 
that $v$ and $e_i$ must be colinear and  that $\kappa_k$ is independent of $k\ne i$, which implies that 
$\kappa$ must be a linear combination of $\hat{q}$ and $v^2$. We conclude:


\begin{prop}
The kernel of $\phi_v$ is the line generated by $v^2-q(v)\hat{q}$, which is nonzero for all $v\in W$.
In particular $\phi: W\ra Hom(S_{[2]}W,S_{[2,1]} W)$ yields a $\frac{m^2+m-2}2\times \frac{m^3-4m}3$ matrix of linear forms of constant rank
$(m^2+m-4)/2$. 
\end{prop}

When $m=3$ we get a $5\times 5$ space of constant rank $4$.  Since  $\fso_3=\fsl_2$ we  {may} see 
the space  in terms of $SL_2$, as $S^2\BC^2\subset \thom(S^4\BC^2,S^4\BC^2)$. The inclusion
on decomposable elements  is  $\ell^2\mapsto (m^4\mapsto (\ell\ww m)\ell m^3)$.
On a rank one element $\ell^2$ the kernel is spanned by $\ell^4$. Let $x,y$ be a
unimodular basis of $\BC^2$,
then at $x^2+y^2$, the kernel is $x^4+y^4-2x^2y^2=(x^2-y^2)^2$. I.e., over all points $v\in S^2\BC^2$, the kernel
is $(v^\perp)^2$. The associated kernel bundle is thus $\cO_{\pp 2}(-2)$, hence $c_1(\cE)=2$.

This space is  a specialization of $\La 2 \BC^5\subset \BC^5\ot \BC^5$ because it corresponds to 
  the representation $\rho: \fso_3\ra \tend(\BC^5)$ with
image in $\fso_5\isom \wedge^2\BC^5$. However, unlike $\La 2 \BC^5$, it is of constant rank.
In fact the $m=3$ case  generalizes to all odd dimensional representations of $\fso_3$, they all map to spaces of
corank one, and are just specializations of the skew-symmetric matrices in odd dimensions.
The interesting point here is that one obtains constant rank matrices.

When $m=4$ we get a $9\times 16$ space of constant rank $8$.  Since  $\fso_4=\fsl_2\times\sl_2$ we  {may} see 
the space  in terms of two spaces $A, B$ of dimension two, with $W=A\otimes B$. Then 
$$S_{[2]}W=S_2A\otimes S_2B, \quad S_{[31]}W=S_3A\otimes S_{21}B\oplus S_{21}A\otimes S_3B,$$
and the resulting morphisms are of the type discussed in Proposition \ref{constant-gen}.

\begin{remark} A related example was studied in 
\cite{MR2348285}, where it was observed that  the unique (up to scale) 
equivariant morphism 
$$\psi : S^3\BC^2\lra Hom(S^{3d}\BC^2,S^{3d+1}\BC^2) {,}
$$
which on powers of linear forms is $m^3\mapsto (\ell^{3d}\mapsto (m\ww \ell)m^2\ell^{3d-1})$, 
has constant  {corank one}  (and this is no longer true
if one replaces $3d$ by some integer not divisible by three). This leads to an interesting 
$SL_2$-equivariant instanton on $\PP^3$. 
\end{remark}

\subsection{Case  $\mu= (3,1,1)$ and $\nu=(3,2,1)$} 
For $v$ non-isotropic, \eqref{kerphiv} gives
$$Ker(\phi_v)  =  \wedge^3\ell^\perp\otimes \ell^2\oplus \wedge^2\ell^\perp\otimes \ell^3.$$
On the other hand, when  $v$ is isotropic, a calculation similar to the previous case gives
$$Ker(\phi_v)  \simeq  \wedge^3H\oplus  {(\wedge^2H)^{\op 2}}\oplus H \simeq  \wedge^3 \ell^\perp\oplus 
\wedge^2\ell^\perp.$$
 {Thus} the two kernels have the same dimension, and  {we}   conclude:

\begin{prop}
The map $\phi: W\ra Hom(S_{[3,1,1]}W,S_{[3,2,1]} W)$ yields a matrix of linear forms of constant corank $\binom{m-1}3+
\binom{m-1}2$. 
\end{prop}
 
 \subsection{Problem: determine which $SO(W)$-inclusions
of the standard representation have constant rank}
  To solve this problem, it is enough to compare the two possible values 
of the rank of $\phi_v$, obtained for $v$ isotropic, or non-isotropic. In the latter case, the analysis 
above allows one  to extend Proposition \ref{kerimcoker} to the orthogonal case, and yields the analogue 
formula for the kernel of $\phi_v$:
\be\label{kerphiv}Ker(\phi_v)  =  \bigoplus_{k\ge 0}\bigoplus_{\mu\stackrel{k}{\ra}\alpha, c\notin\alpha} 
S_{[\alpha]}\ell^\perp\otimes \ell^k.
\ene
Now consider the case where $v$ is isotropic. Then its stabilizer is a parabolic subgroup $P$ of $SO(W)$
and   as in the symplectic case, choosing a Levi subgroup $L$ of $P$ is equivalent to fixing an
orthogonal decomposition  $W=H\oplus (\ell\oplus\ell')$, so that $\ell^\perp=H\oplus\ell$. 
 {A natural approach would be to} try to use a branching formula from $SO(W)$ to $SO(H)\times SO(\ell\oplus\ell')$
to describe the kernel of $\phi_v$, and   to compare  {the result} with \eqref{kerphiv}.

\section{Spin  representations}\label{spinsect}
Let $\Delta_+$ and $\Delta_-$ denote the two half-spin representations of $Spin(2n)$. As before,
$W$ denotes the natural representation, of dimension $2n$. There is a natural 
map $$\phi: W\lra 
Hom(\Delta_+,\Delta_-)$$ and it is well-known that $\phi_v$ is an isomorphism when $v$ is not 
isotropic, while the rank $\phi_v$ is half the dimension of $\Delta_\pm$ when $v$ is isotropic. 
In fact the kernel and cokernel of $\phi_v$ give rise to the spinor bundles on the invariant quadric.

Instead of $\phi$,   consider the following     equivariant morphism arising from the same tensor
 $$\psi: \Delta_+\lra Hom(W,\Delta_-).$$
In order to understand this morphism more concretely, recall that 
the spin representations can be constructed by choosing a decomposition $W=E\oplus F$ into a direct
sum of maximal isotropic spaces. In particular $F$ is naturally identified with the dual of $E$. Then we can
respectively define 
$\Delta_+$ and $\Delta_-$ as the even and odd degree parts in the exterior algebra of $E$. 
The map $\phi$, and equivalently $\psi$, 
is then obtained by letting $E$ act by wedge product, and $F$ by contraction.

\begin{prop} Suppose $n=5$. Then for general  $\delta\in \Delta_+$, the kernel of $\psi_\delta$ is one dimensional.
More precisely, if we decompose  $\delta$ as $(\delta_0,\delta_2,\delta_4)$, where $\delta_k\in \wedge^kE$,
then 
$$Ker(\psi_\delta)=\CC (\delta_2\intprod \delta_4^*\op  (\delta_0\delta_4-\frac{1}{2}\delta_2\ww\d_2)^\#) {.}$$
\end{prop}

Here we identified $\wedge^4E$ with $E^\vee\otimes \det E$, and
$\delta^*$ is the image of $\delta$ under this identification.
  Similarly, using the quadratic form,  $\delta_0\delta_4-\frac{1}{2}\delta_2\ww\d_2$ can be 
considered as an element of $E^\vee\otimes \det E\simeq F\otimes \det E$ and we let $\d^\#$ denote
the image of $\d$ under this identification. We thus get a line in 
$(E\oplus F)\otimes \det E$, which is the same as a line in $W$. 

\proof  
A vector $v=e+f$ is in the kernel of $\psi_\delta$ when the following equations are satisfied:
$$
\delta_0 e+f\intprod\delta_2=0, \quad e\wedge\delta_2+f\intprod \delta_4=0, \quad e\wedge\delta_4=0.
$$
These equations respectively take values in $E$, $\wedge^3E$ and $\wedge^5E\simeq \CC$. 
For $\delta_0\ne 0$ the first equation determines $e$ as a function of $f$. Plugging this relation into 
the second equation, and using the identity $f.(\delta_2\ww \delta_2)=2(f.\delta_2)\wedge\delta_2$, we get the relation
$$
f\intprod (\delta_0\delta_4-\frac{1}{2}\delta_2\ww\d_2)=0.
$$
When $ \delta_0\delta_4-\frac{1}{2}\delta_2\ww\d_2$ is nonzero, such an equation determines $f$ up to a 
unique scalar. Indeed 
$$ \delta_0\delta_4-\frac{1}{2}\delta_2\ww\d_2\in \wedge^4E\isom
E^\vee\simeq F,
$$
 so it can 
be considered as an element $\delta_F$ of $F$, and $f$ must be a multiple of $\d_F$. 
We claim that the first two equations imply the last equation 
$(f\intprod \delta_2)\wedge \delta_4=0$. Indeed, since $\delta_2\wedge \delta_4=0$
for degree reasons, it is equivalent to $(f\intprod\delta_4)\wedge \delta_2=0$. But $f\intprod\delta_4$ is a multiple 
of $(f\intprod\delta_2)\ww \delta_2$ by the first two equations, so our equation reduces to  $(f\intprod\delta_2)\wedge \delta_2^{\ww 2}=0$, or equivalently 
$f\intprod \delta_2^{\ww 3}=0$ , which is trivially verified since $ \delta_2^{\ww 3}$ belongs to $\wedge^6E=0$. 
 \qed

\begin{prop}\label{delrc} $\psi(\Delta_+)$ is rank-critical. 
\end{prop}
 
\proof We proceed as for the proof of Proposition \ref{rank-critical},   applying  the results  of \cite{MR2268360}
and showing that for $ L=\psi(\Delta_+)$,  the space of  rank neutral directions
 $RND(L)$ coincides with $L$. This is particularly easy in this case because of the decomposition
$$ Hom(W,\Delta_-)=\Delta_+\oplus W_{\om_1+\om_4},
$$
where $W_{\om_1+\om_4}$ is the irreducible $SO(W)$-module of highest weight $\om_1+\om_4$ using
fundamental weight notation.
  Were $RND(L)$   strictly bigger than $L$, being a $G$-module it would have to be the whole 
$ Hom(W,\Delta_-)$, which is absurd. 
\qed 

\medskip
 To obtain an explicit matrix,  choose a basis $e_1,\ldots ,e_5$ of $E$ and decompose $\delta=\sum_{|I|\;even}\delta_Ie_I$. 
Then the matrix of $\psi_\delta$ has entry $\pm\delta_I$
on the row indexed $i$   and column indexed $I\cup\{1\}$ when
$1\notin I$, on the row indexed $i^*$   and column indexed $I-\{1\}$ when $1\in I$, and zeroes everywhere else. 
We let  $\theta_m=\pm\delta_{ijk\ell}$, with the  $\pm$   the sign of the permutation $mijkl$ of $12345$.  
This yields the following matrix 

$$M_\delta = \begin{pmatrix}
\delta_\emptyset &0&0&0&0&0&-\delta_{12}&-\delta_{13}&-\delta_{14}&-\delta_{15} \\ 
0&\delta_\emptyset &0&0&0&\delta_{12}&0&-\delta_{23}&-\delta_{24}&-\delta_{25} \\
0&0&\delta_\emptyset &0&0&\delta_{13}&\delta_{23}&0&-\delta_{34}&-\delta_{35} \\
0&0&0&\delta_\emptyset &0&\delta_{14}&\delta_{24}&\delta_{34}&0&-\delta_{45} \\
0&0&0&0&\delta_\emptyset &\delta_{15}&\delta_{25}&\delta_{35}&\delta_{45}&0 \\
\delta_{23}&-\delta_{13}&\delta_{12}&0&0&0&0&0&-\theta_5&\theta_4\\
\delta_{24}&-\delta_{14}&0&\delta_{12}&0&0&0&\theta_5&0&-\theta_3\\
\delta_{25}&-\delta_{15}&0&0&\delta_{12}&0&0&-\theta_4&\theta_3&0\\
\delta_{34}&0&-\delta_{14}&\delta_{13}&0&0&-\theta_5&0&0&\theta_2\\
\delta_{35}&0&-\delta_{15}&0&\delta_{13}&0&\theta_4&0&-\theta_2&0\\
\delta_{45}&0&0&-\delta_{15}&\delta_{14}&0&-\theta_3&\theta_2&0&0\\
0&\delta_{34}&-\delta_{24}&\delta_{23}&0&\theta_5&0&0&0&-\theta_1\\
0&\delta_{35}&-\delta_{25}&0&\delta_{23}&-\theta_4&0&0&\theta_1&0\\
0&\delta_{45}&0&-\delta_{25}&\delta_{24}&\theta_3&0&-\theta_1&0&0\\
0&0&\delta_{45}&-\delta_{35}&\delta_{34}&-\theta_2&\theta_1&0&0&0\\
\theta_1&\theta_2&\theta_3&\theta_4&\theta_5&0&0&0&0&0
\end{pmatrix}$$
The blocking is $(E,F)\times (\La 1 E, \La 3 E,\La 5 E)$.

The image of this matrix in $\CC^{10}$ is  the hyperplane
orthogonal to the vector $h=\sum (h_ie_i+h_i^\vee e_i^\vee)$ 
with 
$$h_i= \sum_{j>i} \delta_{ij}\theta_j-\sum_{j<i} \delta_{ij}\theta_j, \qquad
h_i^\vee = \delta_{\emptyset}\theta_i+\delta_{jk}\delta_{\ell m}-\delta_{j\ell}\delta_{km}+\delta_{jm}\delta_{k\ell}.$$

\medskip\noindent {\it Remark}. For $n=5$ there exists a unique equivariant morphism $$ a: Sym^2(\Delta_+)
\rightarrow W,$$ and the kernel of $\psi_\delta$ is generated by $a(\delta)$ when the latter is 
nonzero. The condition $a(\delta)=0$ is a collection of ten quadratic equations, 
which are the generators of the ideal of 
  the spinor variety $ \SS_{10}\subset  \PP^{15}$. 

From this perspective,   Proposition \ref{delrc} is no surprise if one observes that it is 
related with the minimal resolution of this spinor variety.    {This} minimal resolution was computed in 
\cite{weymanE6} and it has the following form: {
\begin{align*}0&\rightarrow \mathcal{O}_{\PP^{15}}(-8)\rightarrow \mathcal{O}_{\PP^{15}}(-6)^{\oplus 10} 
\rightarrow \mathcal{O}_{\PP^{15}}(-5)^{\oplus 16} \\
& \rightarrow \mathcal{O}_{\PP^{15}}(-3)^{\oplus 16} 
\rightarrow \mathcal{O}_{\PP^{15}}(-2)^{\oplus 10} 
\rightarrow \mathcal{O}_{\PP^{15}}\rightarrow \mathcal{O}_{\BS_{10}}\rightarrow 0.
\end{align*}
}
This shows that the $10\times 16$ matrix $\psi_v$ of linear forms can be interpreted as the matrix of linear syzygies 
between the ten quadrics. (The middle matrix of quadrics is 
also interesting.)

\begin{question}
Do the the larger spinor varieties   have property $N_2$ (meaning that the syzygies between their 
quadratic equations are only linear)?   
\end{question}

What is known is that 
these varieties,  {like all homogeneous varieties,
have ideal generated in degree two}
  {and,}   as in  the case for Grassmannians, the space
of quadratic equations is an irreducible module only in small dimensions. 

\smallskip
For $\SS_{12}\subset \PP^{31}$, the space of quadratic equations is isomorphic with $\fso_{12}\simeq \wedge^2W$, where $W=W_{\omega_1}$ is the natural representation. The space of linear syzygies between these quadrics is the irreducible 
module $W_{\omega_1+\omega_5}$. In particular 
the natural equivariant map $\Delta_+\lra Hom(W_{\omega_1+\omega_5}, \wedge^2W)$ yields a  {$32$-dimensional
space of $352\times 66$-matrices of
bounded rank}.

\smallskip
For $\SS_{14}\subset  \PP^{63}$, the space of quadratic equations is still irreducible, being isomorphic with 
$W_{\omega_3}=\wedge^3W$. The space of linear syzygies between these quadrics is reducible, being isomorphic  with 
$U=\Delta_-\oplus W_{\omega_2+\omega_7}$. In particular 
the natural equivariant map $\Delta_+\lra Hom(U, \wedge^3W)$ yields a $4992\times 364$-matrix of
bounded rank of linear forms in $64$-variables.
\smallskip

\begin{question}
For $n>5$, is the morphism $\psi :\Delta_+\longrightarrow Hom(W,\Delta_-)$  of bounded rank?  
\end{question}

\medskip
In general, the spinor variety $\SS_{2n}\subset \PP(\Delta_+)$ is cut-out by a space of quadratic equations that contains 
 $\wedge^{n-4}W$, with multiplicity one \cite{MR2529169}. Hence there is a unique (up to scale) equivariant map  
 $$
 a: Sym^2(\Delta_+)\lra \wedge^{n-4}W.
 $$ 
 
 \begin{prop}
 There exists a unique (up to scale) nonzero equivariant morphism 
 $$\psi : \Delta_+ \ra Hom(W_{\omega_{n-4}}, W_{\omega_{n-5}+\omega_{n-1}}),$$
and this morphism yields a matrix of linear forms of bounded rank.  

Indeed, for any $\delta\in\Delta_+$
we have $$\psi_\delta(a(\delta))=0.$$
 \end{prop}
 
 \proof Recall that $W_{\omega_i}=\wedge^iW$ for $i\le n-2$, and the remaining two 
fundamental representations are  $\Delta_+=W_{\omega_{n}}$ and $ \Delta_-=W_{\omega_{n-1}}$. 
We have also seen that there exist  equivariant morphisms $W\otimes \Delta_\pm\lra \Delta_\mp$.
By duality, we get morphisms $\Delta_\mp\lra \Delta_\mp\otimes W$ (the half-spin representations are
either self-dual, or dual one of the other according to the parity of $n$, but this does not affect 
our conclusion). By \cite[Proposition 3]{MR2529169} there exists a unique component of $\Delta_+\otimes 
W_{\omega_{n-4}}$ isomorphic to $W_{\omega_{n-5}+\omega_{n-1}}$, and this yields the morphism $\psi$. 
Combining it with $a$ we get an equivariant morphism 
$$Sym^3\Delta_+ \lra W_{\omega_{n-5}+\omega_{n-1}}.$$
But according to \cite[Theorem 2]{MR2529169} there is no such morphism! Hence the formula 
$\psi_\delta(a(\delta))=0$ for any $\delta\in\Delta_+$, and consequently $\psi_\delta$ has 
a nontrivial kernel.\qed 


\medskip The matrices we obtain depend on $2^{n-1}$ parameters, and their size $a_n\times b_n$ is also huge.
The Weyl  dimension formula may be used to show:
$$a_n=\binom{2n}{n-4}\simeq \alpha \frac{2^{2n}}{n^{3/2}}, \qquad b_n\simeq \beta \frac{2^{6n}}{n^{31/2}}$$
for some positive constants $\alpha, \beta$. 

\begin{remark} It would be interesting to decide whether the unexpected kernel of this example is one-dimensional,
or bigger.
\end{remark}

\begin{question} The method
 {of using the fact that an equivariant morphism with certain constraints must be zero 
to force another morphism to have non-trivial kernel when it is not expected to}
  seems rather robust, 
 {it} only relies on the vanishing of certain multiplicities 
in tensor products or plethysms. Can it be used  to exhibit other matrices of bounded rank?
\end{question}


\medskip  {The} connection with syzygies is  {not} surprising, since syzygies were already identified 
in \cite{MR954659} as a wide source of examples of spaces of matrices of bounded, or even constant rank. 
We plan to explore this topic  {further.}
  
    \bibliographystyle{amsplain}

\bibliography{Lmatrix}

\end{document}